\documentclass[10pt,a4paper]{article}

\usepackage[left=2cm,right=2cm,top=2cm,bottom=2cm]{geometry}

\usepackage{microtype}
\usepackage{graphicx}
\usepackage{subfigure}
\usepackage{booktabs} 

\usepackage[colorlinks=true,citecolor=blue,linkcolor=blue]{hyperref}

\usepackage{amsmath}
\usepackage{amssymb}
\usepackage{mathtools}
\usepackage{amsthm}

\usepackage{enumerate}

\theoremstyle{plain}
\newtheorem{theorem}{Theorem}[section]
\newtheorem{proposition}[theorem]{Proposition}
\newtheorem{lemma}[theorem]{Lemma}
\newtheorem{corollary}[theorem]{Corollary}
\theoremstyle{definition}
\newtheorem{definition}[theorem]{Definition}
\newtheorem{assumption}[theorem]{Assumption}
\theoremstyle{remark}
\newtheorem{remark}[theorem]{Remark}

\newtheorem*{openquestion}{Open question}

\usepackage[textsize=tiny]{todonotes}
\usepackage{authblk}

\title{Concentration of empirical barycenters in metric spaces}

\author[1]{Victor-Emmanuel Brunel}
\author[2]{Jordan Serres}

\affil[1]{CREST-ENSAE, Victor-Emmanuel.Brunel@ensae.fr}
\affil[2]{CREST-ENSAE, jordan.serres@ensae.fr}

\newcommand{\R}{\mathbb{R}}
\newcommand{\E}{\mathbb{E}}

\newcommand{\var}{\mathrm{Var}}
\newcommand*\diff{\mathop{}\!\mathrm{d}}

\newcommand{\Sp}{\mathbb{S}}
\newcommand{\per}{\mathrm{per}}
\newcommand{\tr}{\mathrm{tr}}

\newcommand{\CAT}{\textrm{CAT}}

\newcommand{\DS}{\displaystyle}

\begin{document}
\maketitle

\begin{abstract}

Barycenters (aka Fréchet means) were introduced in statistics in the 1940's and popularized in the fields of shape statistics and, later, in optimal transport and matrix analysis. They provide the most natural extension of linear averaging to non-Euclidean geometries, which is perhaps the most basic and widely used tool in data science. In various setups, their asymptotic properties, such as laws of large numbers and central limit theorems, have been established, but their non-asymptotic behaviour is still not well understood. In this work, we prove finite sample concentration inequalities (namely, generalizations of Hoeffding's and Bernstein's inequalities) for barycenters of i.i.d. random variables in metric spaces with non-positive curvature in Alexandrov's sense. As a byproduct, we also obtain PAC guarantees for a stochastic online algorithm that computes the barycenter of a finite collection of points in a non-positively curved space. We also discuss extensions of our results to spaces with possibly positive curvature. 

\end{abstract}

\section{Introduction}

Statistics and machine learning are more and more confronted with data that lie in non-linear spaces. For instance, in spatial statistics (e.g., directional data), computational tomography (e.g., data in quotient spaces such as in shape statistics, collected up to rigid transformations), economics (e.g., optimal transport, where data are discrete measures), etc. Moreover, data that are encoded as very high dimensional vectors may have a much smaller intrinsic dimension, for instance, if they are lying on small dimensional submanifolds of the Euclidean space: In that case, leveraging the possibly non-linear geometry of the data can be a powerful tool in order to significantly reduce the dimensionality of the problem at hand. Even though more and more algorithms are developed to work with such data \cite{LimPalfia14,ohtapalfia,zhang2016first,zhang2018towards}, there are still very little theory for uncertainty quantification, especially in non-asymptotic regimes, which are pervasive in machine learning. In this work, we prove statistical results for barycenters of data points, which are the most natural extension of linear averaging to the context of non-linear geometries. Namely, we extend the notion of sub-Gaussian random variables to the case of metric spaces and, under a non-positive curvature condition, we prove versions of both Hoeffding and Bernstein concentration inequalities. Finally, we discuss extensions of our results to the case of metric spaces with possibly positive curvature.

A complete metric space $(M,d)$ is called non-positively curved (NPC for short) if for all pairs $(x,y)\in M$, there exists $m\in M$ satisfying the following: 
\begin{equation}\label{eq:DefNPC}
    d(z,m)^2\leq \frac{1}{2}(d(z,x)^2+d(z,y)^2-\frac{1}{2}d(x,y)^2),
\end{equation}
for all $z\in M$. NPC spaces are known to induce a lot of regularity. Namely, any two points $x,y\in M$ are connected by a unique (constant speed) geodesic $\gamma_{x,y}$, i.e., a continuous mapping $\gamma=\gamma_{x,y}:[0,1]\to M$ satisfying $d(\gamma(s),\gamma(t))=|s-t|d(x,y)$, for all $s,t\in[0,1]$. Geodesics generalize line segments in Euclidean spaces, as shortest paths from one point $x$ to another point $y$. Moreover, the point $z$ in \eqref{eq:DefNPC} is unique, given by the midpoint of $x$ and $y$, i.e., $z=\gamma_{x,y}(1/2)$. Finally, the distance function $d$ is geodesically convex jointly in both variables and for all $x_0\in M$, the function $\frac{1}{2}d(x_0,\cdot)^2$ is $1$-strongly geodesically convex. A function $f$ is called geodesically convex (resp. strongly geodesically convex) if it is convex (resp. strongly convex) along any geodesic $\gamma_{x,y}, x,y\in M$, i.e., $f(\gamma_{x,y}(t))\leq (1-t)f(x)+tf(y)$, for all $x,y\in M$ and $t\in [0,1]$. We refer the reader to \cite{sturm03,bridson2013metric} for more information on NPC spaces and to \cite{bacak2014convex} for a detailed account on convexity in metric spaces. Just to cite a few examples, the family of NPC spaces include: 
\begin{itemize}
    \item Euclidean and Hilbert spaces;
    \item Hyperbolic spaces and, more generally, Cartan-Hadamard manifolds, i.e. simply connected, complete, Riemannian manifolds with everywhere non-positive sectional curvature \cite[Proposition 3.1]{sturm03};
    \item Metric trees: These are trees embedded in the Euclidean plane, in which the distance between any two points is the Euclidean length of the (unique) shortest path between them;
    \item The space $\mathcal S_p$ of symmetric positive definite matrices of size $p$ (for any fixed integer $p\geq 1$) equipped with the following distance: For any $A,B\in\mathcal S_p$, 
$$d(A,B)=\|\log(A^{-1/2}BA^{-1/2})\|_{\textsf{F}},$$
where $\|M\|_{\textsf{F}}=\sqrt{\text{Tr}(MM^\top)}, \forall M\in\R^{p\times p}$, is the Fröbenius norm. This construction is important in the study of matrix geometric means, as we will explain below, in Section~\ref{sec:GM}. For a detailed account on this space and matrix geometric means, we refer the reader to \cite{bhatia2006riemannian} and the references therein.
\end{itemize}

A natural way to extend the notion of averaging from linear to metric spaces is through the notion of barycenters. Given $x_1,\ldots,x_n\in M$ ($n\geq 1$), a barycenter of $x_1,\ldots,x_n$ is any minimizer $b\in M$ of $\sum_{i=1}^{n}d(x_i,b)^2$. One can easily check that if $(M,d)$ is a Euclidean or Hilbert space, the minimizer is unique and it is given by the average of $x_1,\ldots,x_n$; Note, however, that this is not true in general normed spaces. In NPC spaces, a barycenter of $x_1,\ldots,x_n$ is thus a solution to a strongly convex optimization problem, so it always exists and it is unique. More generally, given a probability distribution $\mu$ with two moments on $(M,d)$, one can define the barycenter of $\mu$ as the unique minimizer $b\in M$ of $\int_M d(x,b)^2\diff\mu(x)$. Here, we say that $\mu$ has two moments if and only if the function $d(x_0,\cdot)^2$ is integrable with respect to $\mu$ for some (and hence, by the triangle inequality, for all) $x_0\in M$. It is easy to check that for all pairs $x,y\in M$ and for all $t\in [0,1]$, the barycenter of $(1-t)\delta_x+t\delta_y$ (aka weighted barycenter of $x$ and $y$, with respective weights $1-t$ and $t$) is simply given by $\gamma_{x,y}(t)$. For instance, the barycenter of $x$ and $y$ is the midpoint of $x$ and $y$. The barycenter of a probability distribution can be interpreted as an extension of the notion of expectation in non-linear spaces. Fortunately, defining the barycenter of $\mu$ actually only requires that $\mu$ has one moment, and its barycenter is a minimizer $b$ of $\int_M (d(x,b)^2-d(x,b_0)^2)\diff\mu(x)$ for any fixed $b_0\in M$; One can easily check that the minimizer does not depend on the choice of $b_0$. Another interpretation of the barycenter is the point that minimizes an expected displacement cost as measured with the quadratic distance. This is why barycenters have gained a lot of importance in optimal transport theory in the past 15 years, particularly with applications to economics theory \cite{agueh2011barycenters}.

Barycenters were initially introduced in statistics by \cite{Frechet48} in the 1940's, and later by \cite{karcher1977Riemannian}, where they were better known as Fréchet means, or Karcher means. They were popularized in the field of shape statistics \cite{kendall2009shape} and optimal transport \cite{agueh2011barycenters,cuturi2014fast,le2017existence,claici2018stochastic,kroshnin2019complexity,altschuler2021wasserstein,altschuler2022wasserstein}.


An alternative construction of barycenters, which is computationally much simpler, is defined iteratively as follows.  
Given $n$ points $x_1,\ldots,x_n$, the inductive barycenter is defined as the point $s_n$, where $s_1=x_1, s_2=\gamma_{s_1,x_2}(1/2)$ and for $k=3,\ldots,n$, $s_k=\gamma_{s_{k-1},x_k}(1/k)$.
Barycenters and inductive barycenters do not coincide (unless $M$ is Euclidean or Hilbert) because of the lack of associativity of barycenters in non-linear spaces. Compared to barycenters, inductive barycenters have the advantage that they can be easily updated if the points $x_1,\ldots,x_n$ come sequentially, in an online fashion. Their computation does not imply to solve an optimization problem, and only requires to know geodesics in $M$. However, a drawback of this construction is that it depends on the order of $x_1,\ldots,x_n$. As we will see below, this drawback in very minor in a statistical context, since this dependence vanishes as $n$ grows large and $s_n$ becomes closer and closer to the barycenter of $x_1,\ldots,x_n$, in a stochastic framework. Moreover, as pointed out in \cite{ohtapalfia}, the inductive barycenter of $x_1,\ldots,x_n$ is the output of an online proximal gradient algorithm to approximate their barycenter, when the $x_i$'s are available one at a time. Indeed, for $i=1,\ldots,n$, let $f_i=d(\cdot,x_i)^2$. Then, the barycenter of $x_1,\ldots,x_n$ minimizes the convex function $f=f_1+\ldots+f_n$ and an online proximal gradient algorithm which starts at $y_1=x_1$ outputs, at each step $t=2,\ldots,n$, the point $y_t\in M$ which minimizes $f_t(y)+\frac{1}{2\lambda_t} d(y,y_{t-1})^2$, $y\in M$, where $\lambda_t$ is a varying stepsize, here, $\lambda_t=(2t)^{-1}$.

Let $\mu$ be a probability measure in $(M,d)$ with two (or, again, one would suffice) moments and let $b^*$ be its barycenter. Let $X_1,\ldots,X_n$ be i.i.d. random variables in $M$ with distribution $\mu$. Here, $n\geq 1$ is the sample size and is fixed. We are interested in the estimation of $b^*$ based on $X_1,\ldots,X_n$. Let $\hat b_n$ be their barycenter (referred to as empirical barycenter) and $S_n$ be their inductive barycenter.

Existence and uniqueness of barycenters are, in general, hard problems \cite{afsari2011riemannian, Yokota16,Yokota18}. Asymptotic theory is well understood for empirical barycenters in various setups, particularly laws of large numbers \cite{ziezold1977expected} and central limit theorems in Riemannian manifolds (a smooth structure on $M$ is a natural assumption in order to derive central limit theorems) \cite{bhattacharya2003large,bhattacharya2005large,bhattacharya2017omnibus,eltzner2019smeary,eltzner2019stability}. Only very few non-asymptotic results have been proven so far, most of which hold under fairly technical conditions. First, \cite[Theorem 4.7]{sturm03} bounded the expected value of $d(S_n,b^*)^2$ in NPC spaces. Namely, $\E[d(S_n,b^*)^2]\leq \frac{\sigma^2}{n}$ where $\sigma^2=\E[d(X_1,b^*)^2]$ can be interpreted as the variance of $X_1$. \cite[Corollary 11]{fastconv} provides the same inequality for $\hat b_n$, under the extra constraint that $(M,d)$ has curvature bounded from below. At a high level, this means that the space $(M,d)$ is not branching (i.e., a geodesic cannot split, unlike, for instance, in metric trees) and this ensures some regularity of the tangent cones of $M$. They also extended their result to spaces $(M,d)$ that are not necessarily NPC, but that satisfy a so-called hugging condition. However, except for NPC spaces, there is no explicit metric space that satisfy such a condition. Several non-asymptotic, high probability bounds are also known for empirical and inductive barycenters. \cite{fastconv} proposes a definition of sub-Gaussian random variables (eq. (3.10)), closely related to the one we give below, and proves (Theorem 12), under the hugging condition mentioned above, a nearly sub-Gaussian tail bound (with a residual term that decays exponentially fast with $n$) for the empirical barycenter of i.i.d. sub-Gaussian random variables around the population barycenter $b^*$. \cite{convrate} obtain concentration inequalities for the empirical barycenter $\hat b_n$ of i.i.d., bounded random variables with non-parametric rates, under some metric entropy conditions on $(M,d)$, some of which being similar in spirit to requiring $M$ to have finite dimension. Finally, most closely related to our work, \cite{Funano10} proves a Hoeffding-type inequality for the inductive barycenter $S_n$ of i.i.d., bounded random variables in NPC spaces, with particular focus on metric trees and finite dimensional Hadamard manifolds. In our Corollary \ref{cor:Hoeffding} below, we improve his results in two ways. First, we significantly improve the constants. Second, our bound contains two terms. First, a bias term, which does not contain the probability level $\delta$. This is the only term where the dimension of the space may appear, through the variance $\sigma^2$ of $X_1$. Hence, the second term, which is the sub-gaussian term (since it is proportional to $\sqrt{\log(1/\delta)}$), is decoupled from the variance of $X_1$.

The paper is organized as follows. In Section \ref{sec:Laplace}, we introduce tools from the concentration of measure theory that will allow us to derive non-asymptotic bounds for convergence of empirical barycenters. In Section \ref{sec:main}, we derive Hoeffding and Bernstein concentration inequalities for empirical barycenters in NPC spaces. In Section \ref{app:noniid}, we extend the previous results in the case of non identically distributed random variables (but with the same theoretical barycenter). Finally, in Section \ref{sec:cat1}, we treat the case of possibly positively curved spaces, and derive an optimal rate concentration bound depending on the radius of the space.

In what follows, for any positive integer $n$ and any metric space $(M,d)$, we denote by $d_1$ (without mention of the dependence on $n$, for simplicity) the $L^1$ product metric on $M^n$, defined as $d_1((x_1,\ldots,x_n),(y_1,\ldots,y_n))=d(x_1,y_1)+\ldots+d(x_n,y_n)$.

\section{The Laplace transform and Sub-Gaussian random variables in metric spaces}\label{sec:Laplace}

\subsection{Laplace transform}
In this section, we gather information on the Laplace transform of probability measures on metric spaces. It will allow us to study precisely the concentration phenomenon of barycenters in NPC spaces and in particular to deal with sub-Gaussian random variables in such spaces.
Let $(M,d)$ be a metric space (not necessarily NPC). Denote by $\mathcal F$ the class of all functions $f:M\to\R$ that are $1$-Lipschitz, i.e., such that $|f(x)-f(y)|\leq d(x,y)$ for all $x,y\in M$. Let $X$ be a random variable in $M$. For all $k\geq 1$, we say that $X$ has $k$ moments if $\E[d(X,x_0)^k]$ is finite, where $x_0$ is any arbitrary point in $M$ (note that this definition does not depend on the choice of $x_0$). The Laplace transform of a random variable $X$ that has at least one moment is defined as (see \cite[Section 1.6]{Ledouxconcentration})
\begin{equation}\label{eq:defLaplace}
\Lambda_X(\lambda) := \sup_{f\in \mathcal F} \E[e^{\lambda(f(X)-\E[f(X)])}], \quad \lambda\geq 0.
\end{equation}
Note that if $X$ has one moment, then so does $f(X)$, for all $f\in\mathcal F$. Let us underline the following property of the Laplace transform, whose proof can be found in \cite[Proposition 1.15]{Ledouxconcentration}.

\begin{lemma}\label{produitdelaplace}
If $X_1,\ldots,X_n$ are independent random variables on $(M,d)$ with at least one moment, then the Laplace transform of the random vector $(X_1,\ldots,X_n)$ in $(M^n,d_1)$  satisfies $$\Lambda_{(X_1,\ldots,X_n)} \leq \Lambda_{X_1}\ldots \Lambda_{X_n},$$ where we recall that $d_1$ is the $L^1$ product metric on $M^n$.
\end{lemma}

Finally, we state the following property of the Laplace transform, when a Lipschitz function is applied to a random variable.

\begin{lemma} \label{lemma:LipLap}
    Let $(M^{(1)},d^{(1)})$ and $(M^{(2)},d^{(2)})$ be metric spaces and $\Phi:M^{(1)}\to M^{(2)}$ be an $L$-Lipschitz function, where $L>0$. Then, for all random variables $X$ in $M^{(1)}$ with at least one moment,
    $$\Lambda_{\Phi(X)}(\lambda)\leq \Lambda_X(\lambda L), \quad \forall \lambda\geq 0.$$
\end{lemma}

\subsection{Sub-Gaussian random variables}

\begin{definition}
A random variable $X$ in $(M,d)$ is called $K^2$-sub Gaussian ($K\geq 0$) if and only if its Laplace transform satisfies \begin{equation}\label{defsousgauss}
\Lambda_X(\lambda)\leq e^{\frac{\lambda^2K^2}{2}}, \quad \forall \lambda\geq 0.
\end{equation}
\end{definition}

This definition of sub-Gaussian random variables is stronger than the standard definition given in Euclidean spaces, where the Laplace transform is only defined as a supremum over linear functions, not over all Lipschitz functions. However, in the context of metric spaces, this definition seems most appropriate because of the lack of linear functions and, as will be seen in Lemma \ref{boundedsubgauss} below, allows us to treat bounded random variables. Moreover, as opposed to the definition suggested in \cite{fastconv}, it does not depend on any reference point in $M$.

Sub-Gaussian random variables are well understood and play a very important role in the Euclidean setup. In particular, they are known to enjoy good concentration properties \cite[Section 2.5]{vershynin2018high}. Note that a sub-Gaussian random variable automatically has infinitely many moments.

The following lemma extends important properties of sub-Gaussian random variables in metric spaces.

\begin{lemma}\label{lemmclassiquesousGauss}
Let $X$ be a random variable in $(M,d)$. The following propositions are equivalent.
\begin{enumerate}[(i)]
\item $X$ is $K^2$-sub-Gaussian $(K>0)$.
\item For all $f\in \mathcal F$, 
$$P(|f(X)-\E[f(X)]|\geq t) \leq 2e^{-t^2/(2K^2)}.$$
\item For all $f\in\mathcal F$, 
$$\E[e^{a(f(X)-\E[f(X)])^2}]\leq 2,$$ 
where $a=\frac{\log 2}{16eK^2}$.
\end{enumerate}
\end{lemma}

In particular, Lemma \ref{lemmclassiquesousGauss} indicates that if $X$ is $K^2$-sub-Gaussian, for some $K>0$, then for any $\delta\in (0,1)$ and for any fixed point $x_0\in M$, it holds with probability at least $1-\delta$ that
$$d(X,x_0)\leq \E[d(X,x_0)]+K\sqrt{2\log(2/\delta)}.$$

The next two propositions show that the sub-Gaussian property is preserved by products and by Lipschitz transformations.


\begin{proposition}\label{produitsousgaussien}
Let $X_1,\ldots,X_n$ be independent random variables in $M$ such that each $X_i$ is $K_i^2$-sub-Gaussian for some $K_i>0$. Then, the $n$-uple $(X_1,\ldots,X_n)$ is $(K_1^2+\ldots+K_n^2)$-sub-Gaussian on the product metric space $(M^n, d_1)$.
\end{proposition}
\begin{proof}
This is a direct consequence of Lemma \ref{produitdelaplace}.
\end{proof}

\begin{proposition}\label{lipsousgauss}
Let $(M^{(1)},d^{(1)})$ and $(M^{(2)},d^{(2)})$ be metric spaces and let $X$ be a random variable in $M_1$. Let $K,L>0$. If $X$ is $K^2$-sub-Gaussian and $\Phi:M^{(1)}\to M^{(2)}$ is $L$-Lipschitz, then $\Phi(X)$ is $(L^2K^2)$-sub Gaussian.
\end{proposition}

\begin{proof}
Let $g:M^{(2)}\rightarrow\R$ be a $1$-Lipschitz function such that $g(\Phi(X))$ is centered, and let $\tilde g:=L^{-1}g\circ\Phi$. Then, $\tilde g:M^{(1)}\rightarrow\R$ is $1$-Lipschitz and $\tilde g (X)  $ is centered, so we have
\begin{align*}
\E e^{\lambda g(\Phi(X))-\E[g(\Phi(X))]} & = \E e^{\lambda L(\tilde g(X)-\E[\tilde g(X)])} \\
& \leq \Lambda_\mu(\lambda L)\\
& \leq e^{\lambda^2L^2K^2/2}.
\end{align*}
We conclude by taking the supremum over all such functions $g$.
\end{proof}

\vspace{-2mm}

Let us conclude this section by stating an important lemma from \cite{Ledouxconcentration}, which, similarly to Hoeffding's lemma for real-valued random variables, indicates that bounded random variables are always sub-Gaussian.

\begin{lemma}\cite[Proposition 1.16]{Ledouxconcentration} \label{boundedsubgauss}
Let $X$ be a bounded random variable in the metric space $(M,d)$, i.e. $d(x_0,X)\leq C$ a.s. for some $x_0\in M$ and $C>0$. Then, $X$ is $4C^2$-sub-Gaussian. 
\end{lemma}

\section{Concentration of empirical barycenters in NPC spaces}\label{sec:main}

\subsection{Lipschitz property of barycenters} \label{subsec:Lip}

The last ingredient in order to prove concentration of empirical barycenters and inductive barycenters in NPC spaces is their Lipschitz property, which we state in the next proposition. Denote by $B_n:M^n\to M$ the function that maps any $n$-uple to its (uniquely defined) barycenter and by $\tilde B_n$ the function that maps any $n$-uple to its inductive barycenter.

\begin{proposition}\label{baryarelip}
Assume that $(M,d)$ is an NPC space. Then, both functions $B_n$ and $\tilde B_n$ are $(1/n)$-Lipschitz, $M^n$ being equipped with the $L^1$ product metric $d_1$. 
\end{proposition}

\begin{proof}
The proof for the inductive barycenter function $\tilde B_n$ follows from a simple induction and it can be found in \cite[Lemma 3.1]{Funano10}.

The $1$-Lipschitz property of $B_n$ follows from two different arguments, which we both give here, because they are both instructive. Let $\mathcal P^1(M)$ be the set of all probability measures on $(M,d)$ with finite first moment. For $\mu\in\mathcal P^1(M)$, let $B(\mu)$ be its barycenter, i.e., the (unique) minimizer $b\in M$ of $\E[d(X,b)^2-d(X,b_0)^2]$, where $X\sim\mu$ and $x_0\in M$ is arbitrarily fixed. 

The first argument follows from \cite[Theorem 3.4]{LimPalfia14}, which provides a deterministic connection between barycenters and their inductive versions. Let $(x_1,\ldots,x_n)$ and $(y_1,\ldots,y_n)$ be two $n$-uples in $M$. We extend these $n$-uples into periodic infinite sequences by setting, for all positive integers $k$, $x_k=x_{r_k}$ and $y_k=y_{r_k}$, where $r_k$ is the unique integer between $1$ and $n$ such that $k-r_k$ is a multiple of $n$. Then, \cite[Theorem 3.4]{LimPalfia14} indicates that $\tilde B_k(x_1,\ldots,x_k)$ and $\tilde B_k(y_1,\ldots,y_k)$ converge to $B_n(x_1,\ldots,x_n)$ and $B_n(y_1,\ldots,y_n)$ respectively, as $k\to\infty$. Moreover, thanks to the $(1/(qn))$-Lipschitz feature of $\tilde B_{qn}$ proved above, one has, for all positive integers~$q$, 
\begin{align*}
& d(\tilde B_{qn}(x_1,\ldots,x_{qn}),\tilde B_{qn}(y_1,\ldots,y_{qn})) \\
& \quad \quad \leq \frac{1}{qn}\sum_{k=1}^{qn} d(x_k,y_k)=\frac{1}{n}\sum_{k=1}^{n} d(x_k,y_k).
\end{align*}
Taking the limit as $q\to\infty$ yields the $(1/n)$-Lipschitz property of $B_n$. 


The second argument uses Jensen's inequality, which implies that the barycenter functional $B$ is contractive over $\mathcal{P}^1(M)$, equipped with the Wasserstein distance $W_1$ \cite[Theorem 6.3]{sturm03}. More precisely, for all probability measures $\mu,\nu\in\mathcal P^1(M)$, it holds that $d(B(\mu),B(\nu))\leq W_1(\mu,\nu)$, where $W_1(\mu,\nu)=\inf_{X\sim\mu,Y\sim\nu}\E[d(X,Y)]$. Now, fix two $n$-uples $(x_1,\ldots,x_n)$ and $(y_1,\ldots,y_n)$ in $M^n$ and set $\mu=n^{-1}\sum_{i=1}^n \delta_{x_i}$ and $\nu=n^{-1}\sum_{i=1}^n \delta_{y_i}$. It is clear that $B(\mu)=B_n(x_1,\ldots,x_n)$ and $B(\nu)=B_n(y_1,\ldots,y_n)$. Moreover, $W_1(\mu,\nu)\leq \frac{1}{n}(d(x_1,y_1)+\ldots+d(x_n,y_n))$, which can be seen by taking the coupling $(X,Y)$ of $\mu$ and $\nu$ such that $P(X=a_i, Y=b_i)=\frac{1}{n}, i=1,\ldots,n$. 
\end{proof}

\begin{remark}
    The above proposition actually only requires that $(M,d)$ is a Buseman space, i.e., a metric space $(M,d)$ such that the metric $d$ is geodesically convex jointly in both variables. This property is weaker than the NPC property, since any NPC space is a Buseman space, but, for instance, strictly convex normed spaces are Buseman but not NPC. 
\end{remark}

\subsection{A concentration inequality for barycenters of sub-Gaussian random variables}

We are now in position to state our first main result, which implies concentration of the empirical barycenter and the inductive barycenter of i.i.d. sub-Gaussian random variables in an NPC metric space. This result is a direct consequence of Propositions \ref{produitsousgaussien} and \ref{lipsousgauss} above.

\begin{theorem}\label{thm1}
Let $(M,d)$ be an NPC space and $X_1,...,X_n$ be independent random variables in $(M,d)$ such that for all $i=1,\ldots,n$, $X_i$ is $K_i^2$-sub-Gaussian, for some $K_i>0$. Then both the empirical and inductive barycenters of $X_1,...,X_n$ are $\frac{K_1^2+\ldots+K_n^2}{n^2}$-sub Gaussian.
\end{theorem}

As a consequence, thanks to Lemma \ref{lemmclassiquesousGauss}, for all $1$-Lipschitz functions $f:M\to\R$ and for all $\delta\in (0,1)$, it holds with probability at least $1-\delta$ that
\begin{equation} \label{eq:bary1}
    f(T_n)\leq \E[f(T_n)]+\bar K_n\sqrt{\frac{\log(1/\delta)}{n}},
\end{equation}
where $T_n$ is either the empirical or the inductive barycenter of $X_1,\ldots,X_n$ and $\bar K_n=\sqrt{K_1^2+\ldots+K_n^2}$. 

Assume that $X_1,\ldots,X_n$ are i.i.d. and let $b^*$ be their (population) barycenter. Denote by $\hat b_n$ their empirical barycenter and by $\tilde b_n$ their inductive barycenter. Then, of particular interest is taking $f=d(\cdot,b^*)$ in \eqref{eq:bary1}. Let $\sigma^2=\E[d(X_1,b^*)^2]$ be the variance of $X_1$. When $T_n$ is the inductive barycenter of $X_1,\ldots,X_n$, it follows from Sturm's law of large numbers \cite[Theorem 4.7]{sturm03} that $\E[d(T_n,b^*)]\leq \frac{\sigma}{\sqrt{n}}$. When $T_n$ is the empirical barycenter of $X_1,\ldots,X_n$, the same inequality holds thanks to \cite[Corollary 11]{fastconv} when the space $(M,d)$ has moreover its curvature bounded from below. Therefore, we get the following corollary.

\begin{corollary} \label{cor:Hoeffding}
    Under the same assumptions and notation as in Theorem \ref{thm1}, let $T_n$ be either the empirical or the inductive barycenter of $X_1,\ldots,X_n$. If $T_n$ is the empirical barycenter $\hat b_n$, further assume that $(M,d)$ has curvature bounded from below. Then, for all $\delta\in (0,1)$, it holds with probability at least $1-\delta$ that 
    $$d(T_n,b^*)\leq \frac{\sigma}{\sqrt n}+K\sqrt{\frac{\log(1/\delta)}{n}}$$
\end{corollary}

We can now state Hoeffding's inequality in global NPC spaces.

\begin{corollary}\label{hoeffding}
Let $(M,d)$ be a NPC space, and let $X_1,...,X_n$ be i.i.d random variables and assume that $X_1\in B(x_0,C)$ almost surely, for some $x_0\in M$ and $C>0$. Let $T_n$ be either the empirical or the inductive barycenter of $X_1,\ldots,X_n$. If $T_n$ is the empirical barycenter $\hat b_n$, further assume that $(M,d)$ has curvature bounded from below. Then, for all $\delta\in (0,1)$, it holds with probability at least $1-\delta$ that 
    $$d(T_n,b^*)\leq \frac{\sigma}{\sqrt n}+2C\sqrt{\frac{\log(1/\delta)}{n}}.$$
\end{corollary}

The left hand side of this concentration inequality contains two terms. The first one is a bias term, which simply controls the expected distance from $T_n$ to $b^*$ and the second one is a stochastic term, which contains the probability level $\delta$. If $(M,d)$ is a Hilbert space (which is a special instance of NPC spaces), this inequality actually reads as
$$d(T_n,b^*)\leq \sqrt{\frac{\tr \Sigma}{ n}}+2C\sqrt{\frac{\log(1/\delta)}{n}},$$
where $\Sigma$ is the covariance operator of $X_1$. For instance, if $M=\R^p$ equipped with the Euclidean structure and $\Sigma$ is the identity matrix, then $\tr \Sigma=p$ and the dimension of $M$ only appears in the bias term, not the stochastic one. This is why Corollary \ref{cor:Hoeffding} is a significant improvement over Funano's result \cite{Funano10}.

As a consequence of Corollary \ref{cor:Hoeffding}, we get the following algorithmic PAC guarantee for the computation of the barycenter of $n$ fixed points $x_1,\ldots,x_n$ ($n\geq 1$) in an NPC space $(M,d)$. Fix $\varepsilon>0$ and $\delta\in (0,1)$. Sample $m$ integers $I_1,\ldots,I_m$ independently, uniformly at random between $1$ and $n$, and set $X_1=x_{I_1}, \ldots, X_m=x_{I_m}$, which are i.i.d. random variables with distribution $\mu=n^{-1}\sum_{i=1}^n\delta_{x_i}$. Finally, iteratively compute $\tilde b_m$, the inductive barycenter of $X_1,\ldots,X_m$.

\begin{corollary} \label{cor:algo}
    Denote by $b^*$ the barycenter of $x_1,\ldots,x_n$. Let $D$ be the diameter of the set $\{x_1,\ldots,x_n\}$. Then, if $m\gtrsim \frac{D^2}{\varepsilon^2}\max(1,\log(1/\delta))$, it holds that $d(\tilde b_m,b^*)\leq \varepsilon$ with probability at least $1-\delta$.
\end{corollary}

The proof of this corollary simply uses the fact that the variance $\sigma^2$ of $\mu$, which is supported on a set of diameter $D$, is bounded by $4D^2$. To the best of our knowledge, Corollary \ref{cor:algo} is the first PAC guarantee for the computation of barycenters in metric spaces. Note that this bound can even be further improved, using Theorem~\ref{thmbernstein} below, see Section~\ref{sec:Bernstein}.

In comparison to this guarantee, \cite[Theorem 3.4]{LimPalfia14} gives a deterministic guarantee for finding an $\varepsilon$-approximation of the barycenter of $x_1,\ldots,x_n$, after $\frac{n(D^2+\sigma^2)}{\varepsilon^2}$ steps: The complexity of their algorithm is $n$ times worse than ours, where $n$ is the number of input points.

Perhaps surprinsingly, the sample complexity bound given in Corollary \ref{cor:algo} does not require any bound on some notion of dimension of the space $M$. However, the algorithm implicitly requires the computation of geodesic segments, at each iteration of the sequence of inductive barycenters $\tilde b_1,\ldots,\tilde b_m$. Such computations have a complexity which, in general, shall depend on the dimension (in a certain sense) of $M$. When $(M,d)$ is a metric tree, the computation of the inductive barycenters simply requires to identify the shortest paths between any two points, which can be done efficiently.

\subsection{An example: Matrix geometric means} \label{sec:GM}

In this section, we consider the computation of the geometric mean of $n$ symmetric, positive definite matrices $A_1,\ldots,A_n\in \mathcal S_p$ ($p\geq 1$). Recall that the geometric mean of $A_1,\ldots,A_p$ is their barycenter, associated with the metric $d(A,B)=\|\log(A^{-1/2BA^{-1/2}}\|_{\textsf{F}}$, for all $A,B\in\mathcal S_p$. It is known that $\mathcal S_p$, equipped with this metric, is an NPC space \cite[Proposition 5]{bhatia2006riemannian}. The geometric mean of two matrices $A,B\in\mathcal S_p$ is the matrix $A\# B=A^{1/2}(A^{-1/2}BA^{-1/2})^{1/2}A^{1/2}$ and more generally, the geodesic segment between $A$ and $B$ is given by $\gamma_{A,B}(s)=A^{1/2}(A^{-1/2}BA^{-1/2})^sA^{1/2}$, also denoted by $A\#_s B$, for all $s\in [0,1]$.

The sequence of inductive barycenters is defined as follows:
\begin{itemize}
    \item $S_1=A_1$;
    \item $S_2=S_1\# A_2=A_1^{1/2}(A_1^{-1/2}A_2A_1^{-1/2})^{1/2}A_1^{1/2}$;
    \item For $k=2,\ldots,m$, $S_k = S_{k-1}\#_{\frac{1}{k}} A_k=S_{k-1}^{1/2}(S_{k-1}^{-1/2}A_kS_{k-1}^{-1/2})^{1/k}S_{k-1}^{1/2}$.
\end{itemize}
Exact computation of each $S_k, k=2,\ldots,m$, requires matrix products and eigendecompositions, whose complexity depends on the size $p$ of the matrices. In fact, there are faster ways to compute good approximations of $A\#_s B$, for $A,B\in\mathcal S_p$ and $s\in [0,1]$, e.g., by using integral representations and Gaussian quadrature: We refer, for instance, to \cite{bhatia2009positive,simon2019loewner} for more details.

\subsection{Bernstein inequality: A refinement in the case of small variance} \label{sec:Bernstein}

In this section, we derive a Bernstein inequality refining Hoeffding's lemma in the case where the $X_i$'s have a small variance, i.e., much smaller than the diameter of their support. 

\begin{theorem}\label{thmbernstein}
Let $(M,d)$ be an NPC space and $X_1,\ldots,X_n$ be i.i.d. random variables in $(M,d)$. Assume that $d(X_1,x_0)\leq C$ almost surely, for some $x_0\in M$ and $C>0$. Then, for all $1$-Lipschitz functions $f:M\to\R$ and for all $\delta\in (0,1)$,
$$f(T_n)\leq \E[f(T_n)]+\min\left(2\sigma\sqrt{\frac{\log(1/\delta)}{n}},\frac{8C\log(1/\delta)}{3n}\right),$$
where $T_n$ denotes either the empirical or the inductive barycenter of $X_1,...,X_n$ and $\sigma^2=\E[d(X_1,b^*)^2]$ denotes the variance of $X_1$, with $b^*$ the barycenter of the distribution of $X_1$.
\end{theorem}

When the $X_i$'s have a very small variance, i.e., significantly smaller than the diameter $C$ of their support, this inequality greatly improves Corollary \ref{cor:Hoeffding}. By taking $f=d(\cdot,b^*)$, we obtain the following version of Bernstein's inequality in NPC spaces.

\begin{corollary} \label{cor:bernstein}
    With the same assumptions and notation as in Theorem \ref{thmbernstein}, let $T_n$ be either the empirical or the inductive barycenter of $X_1,\ldots,X_n$. If $T_n$ is the empirical barycenter $\hat b_n$, further assume that $(M,d)$ has curvature bounded by below. Then, with probability at least $1-\delta$, it holds that 
    $$d(T_n,b^*)\leq \frac{\sigma}{\sqrt{n}}+\max\left(2\sigma\sqrt{\frac{\log(1/\delta)}{n}},\frac{8C\log(1/\delta)}{3n}\right).$$
    
\end{corollary}

\begin{proof}[Proof of Theorem \ref{thmbernstein}]

For all $\lambda\geq 0$ and $f\in\mathcal F$, let $\psi(\lambda,f)=\log\E[e^{\lambda(f(X_1)-\E[f(X_1)])}]$, for $\lambda\geq 0$ and $f\in\mathcal F$. Using the inequality $\log(u)\leq 1-u$, for all $u>0$, it holds that 
\begin{align*}
    & \psi(\lambda,f) \leq \\
    & \quad \E \left[e^{\lambda(f(X_1)-\E[f(X_1)])} -1 -\lambda(f(X_1)-\E f(X_1))\right],
\end{align*}
for all $\lambda\geq 0$ and $f\in\mathcal F$. Since $X_1\in B(x_0,C)$ almost surely, we have that $f(X_1)-\E f(X_1)\leq 2C$, for all $f\in\mathcal F$. Therefore, since the map $u\mapsto \frac{e^u -1 -u}{u^2}$ is increasing, we obtain that
\begin{align*}
    \psi(\lambda,f) & \leq \E \left[ \frac{e^{2\lambda C} -1 -2\lambda C}{4C^2}(f(X_1)-\E f(X_1))^2\right] \\
    & = \frac{e^{2\lambda C} -1 -2\lambda C}{4C^2}\var(f(X_1)).
\end{align*}
Moreover, for all $f\in\mathcal F$, $\var(f(X_1))\leq \sigma^2$. Indeed, 
\begin{align}
    \var(f(X_1)) & = \E\left[(f(X_1)-\E[f(X_1)])^2\right] \nonumber \\
    & \leq \E\left[(f(X_1)-f(b^*))^2\right] \label{eq:step0} \\
    & \leq \E\left[d(X_1,b^*)^2\right] \nonumber \\
    & = \sigma^2, \nonumber
\end{align}
where \eqref{eq:step0} follows from the fact that $\E\left[(f(X_1)-\E[f(X_1)])^2\right]\leq \E\left[(f(X_1)-a)^2\right]$ for all $a\in\R$, and in particular for $a=f(b^*)$.

By Lemma \ref{produitdelaplace}, it follows that 
$$\Lambda_{X_1,\ldots,X_n}(\lambda)\leq e^{\frac{n\sigma^2}{4C^2}(e^{2\lambda C} -1 -2\lambda C)},$$
for all $\lambda\geq 0$.
Now, let $T_n$ be either the empirical or the inductive barycenter of $X_1,\ldots,X_n$. It follows from Proposition \ref{baryarelip} and Lemma \ref{lemma:LipLap} that
$$\Lambda_{T_n}(\lambda)\leq e^{\frac{n\sigma^2}{4C^2}(e^{2\lambda C/n} -1 -2\lambda C/n)},$$
for all $\lambda\geq 0$.



By Chernoff's bound, we obtain, for all $f\in\mathcal F$,
\begin{align*}
& P\left( f(T_n)-\E f(T_n)\geq t \right)\\
& \leq e^{-\lambda t}\E e^{\lambda(f(T_n)-\E f(T_n))}\\
&\leq e^{-\lambda t + \log\Lambda_{T_n}(\lambda)}\\
&\leq \exp\left[-\lambda t + \left(e^{\frac{2\lambda C}{n}} -1 -\frac{2\lambda C}{n}\right)\frac{n\sigma^2}{4C^2}\right].
\end{align*}
By optimizing in $\lambda\geq 0$, we find the best one to be $\lambda=\frac{n}{2C}\log\left(1+\frac{2Ct}{\sigma^2}\right)$ and therefore,
$$P\left( f(T_n)-\E f(T_n)\geq t \right) \leq \exp\left(-\frac{n\sigma^2}{4C^2}\, h\left(\frac{2Ct}{\sigma^2}\right)\right)$$
where $h(u)=(1+u)\log(1+u) -u$, for all $u\geq 0$.
Since $h(u)\geq \frac{u^2}{2(1+\frac{u}{3})}$ for all $u\geq 0$, we obtain
$$P\left( f(T_n)-\E f(T_n)\geq t \right) \leq \exp\left(-\frac{nt^2}{2(\sigma^2+\frac{2Ct}{3})} \right),$$
which yields the desired result.
\end{proof}


\begin{remark}
    In the above concentration inequalities, we assume that the space $(M,d)$ has curvature bounded by below, when we consider empirical barycenters. The reason is to be able to apply the bound on the expectation of $d(\hat b_n,b^*)$ proven in \cite[Corollary 11]{fastconv}. Indeed, for general NPC spaces, there is no available upper bound on $\E[d(\hat b_n,b^*)]$ (however, the bound that is available for $S_n$ \cite[Theorem 4.7]{sturm03} is valid in any NPC space). We do believe that the bound $\E[d(\hat b_n,b^*)^2]\leq \frac{\sigma^2}{n}$ is valid in all NPC spaces, even those with no lower bound on the curvature, i.e., those that can be branching. The simplest and most natural example of an NPC space with no lower curvature bound is a metric tree. Indeed, such a space is $\CAT(\kappa)$, for all $\kappa\in\R$, since all the geodesic triangles are flat (see Section~\ref{cat} for formal definitions). In a metric tree, barycenters are known to be sticky, i.e., even a large perturbation of some of the points $x_1,\ldots,x_n$ may result in no change of their barycenter \cite[Proposition 5.7]{sturm03}. In particular, barycenters are more rigid in such spaces and empirical barycenters are known to have faster asymptotic rates \cite{hotz2013sticky}. This is why we expect that their behavior should not be worse than in NPC spaces with lower curvature bounds. 

\end{remark}

Theorem \ref{thmbernstein} yields the following algorithmic corollary, which is a refinement of Corollary \ref{cor:algo}.

\begin{corollary} \label{cor:algo2}
    Recall the notation of Corollary \ref{cor:algo}. Let $D$ be the diameter of the set $\{x_1,\ldots,x_n\}$ and $\sigma^2=\min_{b\in M} \frac{1}{n}\sum_{i=1}^n d(x_i,b)^2$. Then, if 
\begin{equation} \label{eq:algoB}
    m\gtrsim \max\left(\frac{\sigma^2}{\varepsilon^2},\frac{D}{\varepsilon}\right)\log(1/\delta),
\end{equation}
it holds that $d(\tilde b_m,b^*)\leq \varepsilon$ with probability at least $1-\delta$.
\end{corollary}

In practice, of course, computing $\sigma^2$ can be costly. However, in \eqref{eq:algoB}, it can be replaced by any upper bound, such as $\frac{1}{n}d(x_i,b)^2$ for any fixed $b\in M$, or $\frac{1}{n^2}\sum_{1\leq i,j\leq n} d(x_i,x_j)^2$. Computing $\frac{1}{n}d(x_i,b)^2$ only requires $n$ computations, but the choice of $b$ can be suboptimal. On the other hand, computing $\frac{1}{n^2}\sum_{1\leq i,j\leq n} d(x_i,x_j)^2$ requires a quadratic (in $n$) number of computations, but yields a tight bound on $\sigma^2$, up to a factor $2$. Indeed, one has the following lemma.

\begin{lemma}
    Let $X$ be a random variable in $M$ with two moments and let $Y$ be an independent copy of $M$. Let $\sigma^2=\min_{b\in M}\E[d(X,b)^2]$. Then, 
$$\sigma^2\leq \E[d(X,Y)^2]\leq 2\sigma^2.$$
\end{lemma}

\begin{proof}
The right hand side is obvious thanks to the triangle inequality (and it does not require independence of $X$ and $Y$). For the left hand side, denote by $F(b)=\E[d(X,b)^2]$, for all $b\in M$. Then, $\sigma^2=\min_{b\in M} F(b)$ and independence of $X$ and $Y$ yields that $\sigma^2\leq F(Y)=\E[d(X,Y)^2|Y]$. The result follows by taking the expectation on both sides. 
\end{proof}

The concentration inequalities given in Theorems \ref{cor:Hoeffding} and \ref{cor:bernstein} can be extended to the case when the $X_i$'s are independent but not identically distributed, which is the object of next section.




\section{Concentration inequalities for non-identically distributed random variables} \label{app:noniid}

In this section, we extend Corollary \ref{cor:Hoeffding} and \ref{cor:bernstein} to the case where $X_1,\ldots,X_n$ are independent, but not identically distributed, and share the same barycenter $b^*$.

\begin{theorem} \label{thm:Hoeffnoniid}
Let $(M,d)$ be a NPC space, and let $X_1,...,X_n$ be independent random variables with same barycenter $b^*\in M$. Further assume that for all $i=1,\ldots,n$, $X_i\in B(x_i,C_i)$ almost surely, for some $x_i\in M$ and $C_i>0$. Let $T_n$ be either the empirical or the inductive barycenter of $X_1,\ldots,X_n$. If $T_n=\hat b_n$, further assume that $M$ has a curvature lower bound. Then, for all $\delta\in (0,1)$, it holds with probability at least $1-\delta$ that 
    $$d(T_n,b^*)\leq \frac{\overline\sigma_n}{\sqrt n}+\overline C_n\sqrt{\frac{\log(1/\delta)}{n}},$$
where $\overline\sigma_n=\sqrt{\frac{\sigma_1^2+\ldots+\sigma_n^2}{n}}$ and $\overline C_n=\sqrt{\frac{C_1^2+\ldots+C_n^2}{n}}.$
\end{theorem}

\begin{proof}
The proof is similar to that of Proposition \ref{cor:Hoeffding}, and the main difference is in bounding the bias term $\E[d(T_n,b^*)]$. When $T_n$ is the empirical barycenter, a close inspection of the proof of \cite[Corollary 11]{fastconv} indicates that the $X_i$'s need not be identically distributed and one readily obtains $\E[d(\hat b_n,b^*)]\leq \frac{\overline\sigma_n}{\sqrt n}$. When $T_n$ is the inductive barycenter, we adapt the proof of \cite[Theorem 4.7]{sturm03} and obtain the following lemma.

\end{proof}

\begin{lemma}
Let $X_1,\ldots,X_n$ be independent random points with two moments in an NPC space $(M,d)$ and with same barycenter $b^*$. Then, 
$$\E[d(S_n,b^*)^2]\leq \frac{\sigma_1^2+\ldots+\sigma_n^2}{n^2},$$
where $\sigma^2=\E[d(X_i,b^*)^2]$ is the variance of $X_i$, for $i=1,\ldots,n$. 
\end{lemma}

\begin{proof}
    The proof of this lemma follows the same lines as the proof of \cite[Theorem 4.7]{sturm03} and proceeds by induction on $n$. Denote by $V_n=\E[d(S_n,b^*)^2]$ and by $F_n(x)=\E[d(X_n,x)^2]$, for all $x\in M$. Then, by the $2$-geodesic strong convexity of the squared distance to any given point, one obtains
    \begin{align*}
        V_n & \leq \E\left[\left(1-\frac{1}{b}\right)d(S_{n-1},b^*)^2+\frac{1}{n}d(X_n,b^*)^2-\frac{n-1}{n^2}d(S_{n-1},X_n)^2\right] \\
        & = \frac{n-1}{n}V_{n-1}+\frac{\sigma_n^2}{n}-\frac{n-1}{n^2}\E[F_n(S_{n-1})],
    \end{align*}
    where we use the fact that $X_n$ and $S_{n-1}$ are independent, by construction of $S_{n-1}$ (which only depends on $X_1,\ldots,X_{n-1}$). Now, again by the $2$-geodesic strong convexity of the squared distance to any given point, we obtain that $F_n$ is also $2$-geodescailly strongly convex, yielding $F(S_n)\geq F(b^*)+d(S_n,b^*)^2$ almost surely, since $b^*$ is the minimizer of $F_n$ by assumption. Therefore, it follows that 
    \begin{align*}
        V_n & \leq \frac{n-1}{n}V_{n-1}+\frac{\sigma_n^2}{n}-\frac{n-1}{n^2}\sigma_n^2-\frac{n-1}{n^2}V_{n-1} \\
        & = \frac{(n-1)^2}{n^2}V_{n-1}+\frac{\sigma_n^2}{n^2}.
    \end{align*}
    In other words, $n^2V_n\leq (n-1)^2V_{n-1}+\sigma_n^2$. Since, by definition of $S_1$, $V_1=\sigma_1^2$, the result follows by induction. 
    
\end{proof}

Finally, we also have the following generalization of Corollary \ref{cor:bernstein} when the $X_i$'s are not identically distributed.

\begin{theorem} \label{thm:Bernsteinnoniid}
    Let $(M,d)$ be an NPC space and $X_1,\ldots,X_n$ be independent random variables in $(M,d)$. Assume that $d(X_i,x_0)\leq C$ almost surely, for all $i=1,\ldots,n$, for some fixed $x_0\in M$ and $C>0$. let $T_n$ be either the empirical or the inductive barycenter of $X_1,\ldots,X_n$. If $T_n$ is the empirical barycenter $\hat b_n$, further assume that $(M,d)$ has curvature bounded by below. Then, for all $\delta\in (0,1)$, it holds with probability at least $1-\delta$ that 
    $$d(T_n,b^*)\leq \frac{\overline\sigma_n}{\sqrt{n}}+\min\left(2\overline\sigma_n\sqrt{\frac{\log(1/\delta)}{n}},\frac{8C\log(1/\delta)}{3n}\right),$$
    where $\overline\sigma_n$ is as in Theorem \ref{thm:Hoeffnoniid}.
\end{theorem}

In Theorem \ref{thm:Hoeffnoniid}, when $X_1,\ldots,X_n$ do not even share the same barycenter, we still have the following fact, for any fixed $b\in M$. For any $\delta\in (0,1)$, it holds with probability at least $1-\delta$ that
$$d(T_n,b)\leq \E[d(T_n,b)]+\overline C_n\sqrt{\frac{\log(1/\delta)}{n}}.$$
However, it is not clear what $b$ to choose in order to make the first term as small as possible. If $b_1,\ldots,b_n$ are the respective population barycenters of $X_1,\ldots,X_n$ (i.e., each $b_i$ minimizes $\E[d(X_i,b)^2]$ over $b\in M$), a natural candidate for $b$ would be the barycenter of $b_1,\ldots,b_n$. 

\begin{openquestion}
Let $X_1,\ldots,X_n$ be independent random variables in $M$, with two moments. For each $i=1,\ldots,n$, let $b_i$ be the barycenter of $X_i$ and $\sigma_i^2=\E[d(X_i,b_i)^2]$ its variance. Is it true that 
$$\E[d(\hat b_n,b_n^*)^2]\leq\frac{\overline\sigma_n^2}{n},$$
where $b_n^*$ is the barycenter of $b_1,\ldots,b_n$ and $\overline\sigma_n^2=\frac{\sigma_1^2+\ldots+\sigma_n^2}{n}$, as in Theorem \ref{thm:Hoeffnoniid}?
\end{openquestion}

\section{Beyond non-positively curved spaces}\label{sec:cat1}

It is natural to ask whether Theorems \ref{thm1} and \ref{thmbernstein} still hold, perhaps with worse constants and/or rates, without the non-positive curvature assumption. A natural framework to do that is the notion of CAT spaces. Briefly, a metric space is said to be $\CAT(\kappa)$, $\kappa\in\R$, when its triangles are thinner than they would be in the model space of curvature $\kappa$ (hyperbolic plane if $\kappa<0$, Euclidean plane if $\kappa=0$ and $2$-dimensional sphere of radius $1/\kappa$ if $\kappa>0$): See \cite{bridson2013metric,burago2022course} for an introduction to CAT spaces. In this terminology, NPC spaces are $\CAT(0)$ spaces. An important example of at $\CAT(\kappa)$ space, which is closely related to optimal transport theory, for $\kappa>0$, is the space of all symmetric positive definite matrices of size $p\geq 1$, with all eigenvalues that are at least $\sqrt{3/(2\kappa)}$ \cite[Proposition 2]{massart2019curvature}, equipped with the Bures-Wasserstein metric. The Bures-Wasserstein metric between any two symmetric positive definite matrices $A,B$ in $\R^{p\times p}$ is given by $W_2(\mathcal N(0,A),\mathcal N(0,B))$, where $W_2$ is the $2$-Wasserstein distance and $\mathcal N(0,A)$ is the $p$-variate centered Gaussian distribution with covariance matrix $A$. We refer the reader to \cite{bhatia2019bures} for more information about the Bures-Wasserstein metric. 
Let us first introduce some background on CAT spaces. 

\subsection{Background on CAT spaces}\label{cat}

In this section, we give the precise definition of CAT spaces. We refer the reader to the book \cite{Alexandrovgeom} for a complete view on the topic. First, we introduce a family of model spaces that will allow us to define local and global curvature bounds in the sequel. Let $\kappa\in\R$. 

\paragraph{$\kappa=0$: Euclidean plane} 
Set $M_0=\R^2$, equipped with its Euclidean metric. This model space corresponds to zero curvature, is a geodesic space where geodesics are unique and given by line segments. 

\paragraph{$\kappa>0$: Sphere}
Set $M_\kappa=\frac{1}{\sqrt\kappa}\Sp^2$: This is the $2$-dimensional Euclidean sphere, embedded in $\R^3$, with center $0$ and radius $1/\sqrt\kappa$, equipped with the arc length metric: $d_\kappa(x,y)=\frac{1}{\sqrt\kappa}\arccos(\kappa x^\top y)$, for all $x,y\in M_\kappa$. This is a geodesic space where the geodesics are unique except for antipodal points, and given by arcs of great circles. Here, a great circle is the intersection of the sphere with any plane going through the origin. 

\paragraph{$\kappa<0$: Hyperbolic space}
Set $M_\kappa=\frac{1}{\sqrt\kappa}\mathbb H^2$, where $\mathbb H^2=\{(x_1,x_2,x_3)\in\R^3:x_3>0, x_1^2+x_2^2-x_3^2=-1\}$. The metric is given by $d_\kappa=\frac{1}{\sqrt{-\kappa}}\textrm{arccosh}(-\kappa\langle x,y\rangle)$, for all $x,y\in M_\kappa$, where $\langle x,y\rangle=x_1y_1+x_2y_2-x_3y_3$. This is a geodesic space where geodesics are always unique and are given by arcs of the intersections of $M_\kappa$ with planes going through the origin. 

\vspace{4mm} 
For $\kappa\in\R$, let $D_\kappa$ be the diameter of the model space $M_\kappa$, i.e., 
$\DS D_\kappa=\begin{cases} \infty \mbox{ if } \kappa\leq 0 \\ \frac{\pi}{\kappa} \mbox{ if } \kappa>0\end{cases}$.



Let $(M,d)$ be a geodesic space, i.e., a metric space where any two points have at least one geodesic between them. The notion of curvature (lower or upper) bounds for $(M,d)$ is defined by comparing the triangles in $M$ with triangles with the same side lengths in model spaces. 

\begin{definition}
	A (geodesic) triangle in $M$ is a set of three points in $M$ (the vertices) together with three geodesic segments connecting them (the sides). 
\end{definition}

Given three points $x,y,z\in S$, we abusively denote by $\Delta(x,y,z)$ a triangle with vertices $x,y,z$, with no mention to which geodesic segments are chosen for the sides (geodesics between points are not necessarily unique, as seen for example on a sphere, between any two antipodal points). The perimeter of a triangle $\Delta=\Delta(x,y,z)$ is defined as $\per(\Delta)=d(x,y)+d(y,z)+d(x,z)$. It does not depend on the choice of the sides. 

\begin{definition}
	Let $\kappa\in\R$ and $\Delta$ be a triangle in $M$ with $\per(\Delta)<2D_\kappa$. A comparison triangle for $\Delta$ in the model space $M_\kappa$ is a triangle $\bar\Delta\subseteq M_\kappa$ with same side lengths as $\Delta$, i.e., if $\Delta=\Delta(x,y,z)$, then $\bar\Delta=\Delta(\bar x,\bar y,\bar z)$ where $\bar x,\bar y,\bar z$ are any points in $M_\kappa$ satisfying 
	$$\begin{cases} d(x,y)=d_\kappa(\bar x,\bar y) \\ d(y,z)=d_\kappa(\bar y,\bar z) \\ d(x,z)=d_\kappa(\bar x,\bar z). \end{cases}$$
\end{definition}

Note that a comparison triangle in $M_\kappa$ is always unique up to rigid motions. We are now ready to define curvature bounds. Intuitively, we say that $(M,d)$ has global curvature bounded from above (resp. below) by $\kappa$ if all its triangles (with perimeter smaller than $2D_\kappa$) are thinner (resp. fatter) than their comparison triangles in the model space $M_\kappa$. 

\begin{definition}
	Let $\kappa\in\R$. We say that $(M,d)$ has global curvature bounded from above (resp. below) by $\kappa$ if and only if for all triangles $\Delta\subseteq M$ with $\per(\Delta)<2D_\kappa$ and for all $x,y\in \Delta$, $d(x,y)\leq d_\kappa(\bar x,\bar y)$ (resp. $d(x,y)\leq d_\kappa(\bar x,\bar y)$), where $\bar x$ and $\bar y$ are the points on a comparison triangle $\bar\Delta$ in $M_\kappa$ that correspond to $x$ and $y$. We then call $(M,d)$ a $\CAT(\kappa)$ (resp. $\CAT^+(\kappa)$) space. 
\end{definition}

\subsection{Barycenters in CAT spaces}

Now, the goal of the next section is to derive some non-asymptotic concentration bounds for the empirical barycenters in $\CAT(\kappa)$ spaces when $\kappa>0$. It is well known that barycenters may no longer be unique in such spaces. For instance, any point on the equator is a barycenter of the North and South poles, on any sphere. However, it is known that if a probability measure on a $\CAT(\kappa)$ space, with $\kappa>0$, is supported within a small enough ball, then it does have a unique barycenter \cite[Theorem B]{Yokota16}. This is related to the following convexity property of the squared distance to any fixed point \cite[Proposition 3.1]{Ohtaconvexity}.

\begin{lemma}
    Let $(M,d)$ be a $\CAT(\kappa)$ space, with $\kappa>0$ and let $x_0\in M$. Then, for all $\varepsilon\in (0,\frac{\pi}{2\sqrt{\kappa}}]$, $\frac{1}{2}d(x_0,\cdot)^2$ is $k_\varepsilon$-geodesically strongly convex on the ball $B\left(x_0,\frac{\pi}{2\sqrt{\kappa}}-\varepsilon\right)$, with $k_\varepsilon=(\pi-2\sqrt{\kappa}\varepsilon)\tan(\varepsilon\sqrt{\kappa})$.
\end{lemma}


This strong convexity property implies a variance inequality in $\CAT(\kappa)$ spaces of small diameter for $\kappa>0$ . 

\begin{lemma}\label{variaceineqcatun}
Let $(M,d)$ be a $\CAT(\kappa)$ space with $\kappa>0$, let $x_0\in M$ and let $\varepsilon\in (0,\frac{\pi}{2\sqrt{\kappa}}]$. Let $X$ be a random variable in $M$ with values in the ball $B\left(x_0,\frac{\pi}{2\sqrt{\kappa}}-\varepsilon\right)$ almost surely. Then it satisfies the following variance inequality 
\begin{equation}
    d^2(z,b^*) \leq \frac{2}{k_\varepsilon}\E \left[d^2(z,X)-d^2(b^*,X)\right], \quad \forall z\in M,  
\end{equation}
where $k_\varepsilon=(\pi-2\sqrt{\kappa}\varepsilon)\tan(\varepsilon\sqrt{\kappa})$, and $b^*$ denotes the Fréchet mean of the law of $X$.
\end{lemma}

Note that $k_\varepsilon\in (0,2)$, for all $\varepsilon\in (0,\frac{\pi}{2\sqrt{\kappa}})$. Moreover $k_\varepsilon\xrightarrow[\varepsilon\to 0]{} 0$, and $k_\varepsilon\xrightarrow[\varepsilon\to \frac{\pi}{2\sqrt{\kappa}}]{} 2$. So when the diameter tends to zero, the convexity tends to be the same as in the Euclidean case, whereas when the diameter tends to be maximal (i.e. $\frac{\pi}{2\sqrt{\kappa}}$), the geodesic strong convexity may no longer hold. 

\begin{proof}
    By \cite[Proposition 3.1]{Ohtaconvexity}, we know that for all $p\in M$, the functions $z\mapsto d^2(z,p)$ are $k_\varepsilon$-convex on the ball $B\left(x_0,\frac{\pi}{2\sqrt{\kappa}}-\varepsilon\right)$. Moreover, $b^*$ is in the ball by \cite[Theorem B]{Yokota16}. It follows that the function $z\mapsto \E\left[ d^2(z,X)-d^2(b^*,X)\right]$ is also $k_\varepsilon$-convex on the ball. Therefore, by taking $z_t$ the joining geodesic between $b^*$ and $z$ (it exists and is unique because the ball is small enough), and for $p=b^*$, by definition of the barycenter, we get that for $t\in[0,1]$,
\begin{align*}    
     0 &\leq \E\left[ d^2(z_t,X)-d^2(b^*,X)\right]\\
     & \leq t\, \E\left[ d^2(z,X)-d^2(b^*,X)\right] -\frac{k_\varepsilon}{2}t(1-t)d^2(z,b^*)    
\end{align*}
Therefore, $$t\,\E\left[ d^2(z,X)-d^2(b^*,X)\right] \geq \frac{k_\varepsilon}{2}t(1-t)d^2(z,b^*),$$
yielding the result by dividing by $t$ and letting $t$ goes to zero.
\end{proof}
In order to derive a non asymptotic bounds on the convergence of the empirical barycenter in $\CAT(\kappa)$ spaces for $\kappa>0$, we have to require the following.

\begin{assumption}\label{assumentropy}
    The metric space $(M,d)$ is $\CAT(\kappa) $ for some $\kappa>0$ and there are positive constants $A,p>0$ such that for all $x\in M$ and for all $\alpha,r>0$ with $\alpha\leq r$,
    \begin{equation}\label{formuleassumentropy}
     N(B(x,r),\alpha) \leq \left(\frac{A r}{\alpha} \right)^p,
    \end{equation}
    where $N(B(x,r),\alpha)$ denotes the smallest integer $N\geq 1$ such that the ball $B(x,r)$ can be covered by $N$ balls of radius $\alpha$.
\end{assumption}

\begin{remark}
    The logarithm of the quantity $N(B(x,r),\alpha)$ is known as the metric entropy of $B(x,r)$ and is used in the theory of empirical processes (see e.g. \cite{surveymetricentropy}).
    If, for instance, $M$ is a Riemannian manifold of dimension $p$ with bounded curvature, then it satisfies \eqref{formuleassumentropy} with the same $p$ as exponent. This can be easily seen by a volumetric argument using estimates of the volumes of balls in Riemannian manifolds, see, e.g., \cite[Section III]{chavel2006riemannian}. Therefore, \eqref{formuleassumentropy} characterizes, in some way, the finiteness of the dimension of $M$. 
\end{remark}
We can now state the following concentration inequality in the framework of $\CAT(\kappa)$ spaces ($\kappa>0$).

\begin{theorem}\label{thmcatun}
    Let $(M,d)$ be a $\CAT(\kappa)$ space, $\kappa>0$, satisfying Assumption \ref{assumentropy}. Let $x_0\in M$ and $\varepsilon\in (0,\frac{\pi}{2\sqrt{\kappa}}]$. Let $X_1,...,X_n$ be i.i.d random variables such that $X_1\in B\left(x_0,\frac{\pi}{2\sqrt{\kappa}}-\varepsilon\right)$ almost surely. Denote by $b^*$ the population barycenter of $X_1$ and let $\hat b_n$ be the empirical barycenter of $X_1,\ldots,X_n$. Then, for all $\delta\in (0,1)$, it holds with probability at least $1-\delta$ that

    $$ d(\hat b_n,b^*) \leq c\left(\frac{A}{\varepsilon\kappa}\sqrt{\frac{p}{n}}+ \frac{1}{\varepsilon^{3/2}\kappa}\sqrt{\frac{\log(2/\delta)}{n}} \right),$$
    
    where $c>0$ is a universal constant and $A$ is the constant from Assumption \ref{assumentropy}.
\end{theorem}
\begin{proof}
The proof is based on an application of \cite[Theorem 2.1]{convrate}. It is possible to use this theorem thanks to Assumption \ref{assumentropy} and Lemma \ref{variaceineqcatun}. By a careful analysis of its proof with the constants $D=p $, $C=A$, $K_1=\sqrt{\frac{\pi}{2\sqrt{\kappa}}-\varepsilon}$, $K_2=2\left(\frac{\pi}{2\sqrt{\kappa}}-\varepsilon\right) $, $K_3=\frac{2}{(\pi-2\sqrt{\kappa}\varepsilon)\tan(\varepsilon\sqrt{\kappa})}$, and $\alpha=\beta=1 $, we obtain that for all $\delta\in(0,1)$, with probability at least $1-\delta$,
$$\sqrt{\frac{k_\varepsilon}{2}}d(\hat b_n,b^*) \leq 3c_1\sqrt{\frac{p}{n}} + 3c_2\sqrt{\frac{\log(2/\delta)}{n}},$$ where $$k_\varepsilon= (\pi-2\sqrt{\kappa}\varepsilon)\tan(\varepsilon\sqrt{\kappa}),$$ $$c_1=\frac{96\sqrt{2A}\left(\frac{\pi}{2\sqrt{\kappa}}-\varepsilon\right)}{\sqrt{(\pi-2\sqrt{\kappa}\varepsilon)\tan(\varepsilon\sqrt{\kappa})}} = \frac{96\sqrt A\sqrt{\frac{\pi}{2\sqrt\kappa}-\varepsilon}}{\sqrt{2\kappa\tan(\varepsilon\sqrt{\kappa})}}$$ and $$c_2 = \frac{4\left(\frac{\pi}{2\sqrt{\kappa}}-\varepsilon\right)}{\sqrt{(\pi-2\sqrt{\kappa}\varepsilon)\tan(\varepsilon\sqrt{\kappa})}} + \frac{16}{3}\sqrt{\left(\frac{\pi}{2\sqrt{\kappa}}-\varepsilon\right)} = \sqrt{\frac{\pi-2\varepsilon\sqrt{\kappa}}{\kappa}}\left( \frac{2}{\sqrt{\tan(\varepsilon\sqrt{\kappa}})}+\frac{16}{3\sqrt{2}}\right). $$
Hence, with probability at least $1-\delta$, one has
$$d(\hat b_n,b^*)\leq \frac{288 \sqrt A}{\sqrt{\kappa}\tan(\varepsilon\sqrt{\kappa})}\sqrt{\frac{p}{n}}+ \frac{1}{\sqrt{\kappa\tan(\varepsilon\sqrt{\kappa})}}\left( \frac{6\sqrt{2}}{\sqrt{\tan(\varepsilon\sqrt{\kappa})}}+16\right)\sqrt{\frac{\log(2/\delta)}{n}}, $$ which proves the result by noting that $\tan x\geq x$ for all $x\in [0,\pi/2)$. 
\end{proof} 

\textbf{Acknowledgments:} {The second author would like to thanks Miklós Pálfia and Takomi Yokota for useful discussion.}

\bibliographystyle{plain}
\bibliography{biblio}

\end{document}